\documentclass[12pt,a4paper]{article}
\usepackage{a4wide}

\usepackage{amsmath}
\usepackage{amsthm}
\usepackage{amsfonts}

\newtheorem{thm}{Theorem}%[section]
\newtheorem{lem}{Lemma}%[section]
\newtheorem{property}{Property}%[section]
\theoremstyle{remark}
\newtheorem{rem}{Remark}%[section]
\newtheorem{example}{Example}%[section]

\newcommand{\R}{\mathbb{R}}
\newcommand{\Rp}{\mathbb{R}_+}

\parskip \smallskipamount

\title{Asymptotics for the maximum of a modulated random walk
with heavy-tailed increments}
\author{Serguei Foss and Stan Zachary}
\date{}

\begin{document}

\maketitle

\begin{abstract}\noindent
  We consider asymptotics for the maximum of a modulated random walk
  whose increments~$\xi_n^{X_n}$ are heavy-tailed.  Of particular
  interest is the case where the modulating process~$X$ is
  regenerative.  Here we study also the maximum of the recursion given
  by $W_0=0$ and, for $n\ge1$, $W_n=\max(0,W_{n-1}+\xi_n^{X_n})$.
\end{abstract}

\section{Introduction}

Let $S_n = \sum_1^n \xi_i$ be a sum of i.i.d.\ random variables (r.v.s)
with a negative finite mean ${\bf E} \xi_1 = -a <0$. The common
distribution of the random variables~$\xi_n$ is assumed to be right
heavy-tailed (i.e.  ${\bf E} \exp (\lambda \xi_1 ) = \infty$ for all
$\lambda > 0$).  Moreover, the second tail of this distribution is
assumed to be subexponential (see
Section 2 for definitions). Then the classical result (see, e.g.,
\cite{Ver}) states that, as $y\to\infty$,
\begin{equation} \label{11}
{\bf P} (\sup_n S_n > y) \sim \frac{1}{a} \int_y^{\infty} {\bf P} 
(\xi_1>t) dt .
\end{equation}

We consider here a more general random walk 
\begin{equation} \label{22}
S_n = \sum_1^n \xi_i^{X_i}
\end{equation}
whose increments $\xi^{X_n}_n$ are modulated by an independent
sequence $X=\{X_n\}_{n\ge1}$ (see Section~\ref{sec:results} for more
precise definitions and notation). We assume that $S_n\to -\infty $
a.s.\ and find conditions which are sufficient for the probability of
the ``rare'' event ${\bf P} (\sup_n S_n >y)$ to behave asymptotically
(as $y\to\infty$) similarly to \eqref{11}. The results obtained may be
applied to the study of complex stochastic models with modulated
input.

Particular cases, with $X$ a finite Markov chain, were
considered in \cite{Ar} and \cite{AS}. S.~Asmussen (\cite{SA2})
proposed an approach for getting the asymptotics for ${\bf P} (\sup_n
S_n >y)$ on the basis of a regenerative structure: if the maximum of
the partial sums over a typical cycle behaves asymptotically as the
end-to-end sum, and these asymptotics are subexponential, then the
result \eqref{11} stays the same. In \cite{SAM}, the authors assumed
that $X$ is countably-valued, a certain dependence between the
$X_n$ and the $\xi^x_n$ was allowed, and some homogeneity in $x$ of
the distributions of the random variables~$\xi^x_n$ was required. By
the use of matrix-analytic methods, they found the asymptotics for the
stationary distribution of a Markov chain with increments
$\xi_n^{X_n}$.

In \cite{BSS}, upper and lower bounds were found for the asymptotics
of ${\bf P} (R>y)$, as $y\to\infty$, where $R$ is the stationary
response time in a tandem queue.  Then, in \cite{HS}, the asymptotics
for the stationary waiting time $W$ in the second queue were studied.
Note that $W$ may be represented as the limit of a recursion
$$
W_n = \max (0, W_{n-1} + \xi_n^{X_n})
$$
where $X=\{X_n\}$ forms a Harris ergodic Markov chain.  In \cite{BF},
the exact asymptotics for ${\bf P} (R>y)$ were found.  The proof is
based on ideas similar to that of Lemma~\ref{lem2} of the present
paper.

Finally, nice overviews on the current state of large deviations
theory in the presence of heavy tails were given in \cite{Sig} and in
recent new books \cite{EKM} and \cite{SA}. 

We state our main results in Section~\ref{sec:results}.  We consider
in particular the case where the modulating process~$X$ is
regenerative, where we give also an instructive example and
counterexample.  The latter shows our conditions on the tail of the
distribution of the regeneration time to be best possible---in a sense
made clear there.  We study also the queueing theory recursion given
by $W_0=0$ and, for $n\ge1$, $W_n=\max(0,W_{n-1}+\xi_{n}^{X_n})$.

In Section~\ref{sec:properties} we collect together some useful known
results, most of which are required for our proofs.  These are given in
Section~\ref{sec:proofs}.  Perhaps the key result of the entire paper
is Lemma~\ref{lem2} of that section, which develops an idea found also
in \cite{BF}.

\section{The main results}
\label{sec:results}

Let $({\cal X},{\cal B})$ be a measurable space and $X= \{ X_n
\}_{n\ge1}$ an $\cal X$-valued discrete-time random process.
Let $P : {\cal X} \times {\cal B}_0 \to [0,1]$ (where ${\cal B}_0$ is
the Borel $\sigma$-algebra on $\R$) be a function such that
\begin{itemize}\itemsep 0pt
\item[(i)] for every $x\in {\cal X}$, $P(x,\cdot )$ is a probability
  measure;
\item[(ii)] for every $B\in {\cal B}_0$, $P(\cdot , B)$ is a measurable
  function.
\end{itemize}
For each $x\in {\cal X}$, let $F_x$ denote the distribution
function of $P(x,\cdot)$.  For each integer~$n\ge1$, introduce the
family of real-valued random variables
$\{\xi^x_n\}_{x\in{\cal{}X}}$.  Assume that these families are
mutually independent (in $n$), do not depend on the process $X$, and
that, for each $x\in {\cal X}$ and each $n$, $\xi^x_n$ has
distribution function~$F_x$.  
%\sz{Could define $F_x$ later!}
We define the \emph{random walk~$\{S_n\}_{n\ge0}$ modulated by the
  process~$X$} by $S_0=0$ and, for any $n=1,2,\ldots$,
$$
S_n = \sum_{i=1}^n \xi_i^{X_i}.
$$
Define also, for $n\ge1$, 
$$
M_n = \max_{0\leq i\leq n} S_i,
$$
and let
$$
M = \sup_{n\geq 0} S_n.
$$
Further, for each $y>0$, define
$$
\mu (y) = \min \{ n\geq 1: \ S_n > y \}.
$$
Note that $\mu (y) = \infty $ if and only if $M \leq y$.

We are interested the asymptotics of the upper-tail distribution of
$M$ under conditions which guarantee that the random walk~$S_n$
behaves sufficiently regularly and has a strictly negative drift, and
where additionally the distribution functions~$F_x$ have, in some
appropriate sense, heavy positive tails.  More precisely, we wish to
make statements, under such conditions, about the behaviour, for any
$B\in\cal B$ and as $y\to\infty$, of ${\bf P}(M>y,\,X_{\mu(y)}\in B)$.

Motivated by queueing theory applications, we are also interested in
the behaviour of the process~$\{W_n\}_{n\ge0}$ defined recursively by
$W_0=0$ and, for $n\ge1$,
\begin{equation}\label{wrec}
  W_n = \max (0, W_{n-1} + \xi_{n}^{X_n}).
\end{equation}

We assume throughout that $P$ is such that each probability
measure~$P(x,\cdot)$ (i.e.\ each distribution $F_x$) has a finite
mean.  We further assume throughout that there exist a distribution
function $F$ on $\Rp$ with finite mean, and a measurable function
$c:{\cal X}\to\Rp$ such that
\begin{gather}
  \overline{F}_x(y)  \sim c(x) \overline{F}(y) \quad \text{ as
    $y\to\infty$, \quad for all $ x\in{\cal X}$}, \label{tail}\\
  \sup_x \sup_{y\ge0}\frac{\overline{F}_x(y)}{\overline{F}(y)} = L,
  \qquad\text{ for some $L<\infty$}.
  \label{comp} 
\end{gather}
Here, for any distribution function~$H$ on $\R$, $\overline{H}$
denotes the tail distribution given by $\overline{H}(y)=1-H(y)$.

The following two conditions on the pair $(X,\,P)$ will be satisfied
in all our results, either by hypothesis or as a consequence of more
fundamental modelling assumptions.  (We show below that these
conditions may arise naturally in the case where the process $X$ is
regenerative, but they may also arise in other contexts.)
\begin{itemize}
\item[(C1)] There exists some probability distribution~$\pi$ on
  $({\cal X, B})$ such that, for some positive integer $d$,
  \begin{equation}
    \label{cvge}
    \frac{{\bf P} (X_{n}\in \cdot ) + 
      \ldots + {\bf P} (X_{n+d-1}\in \cdot )}{d}
    \to \pi (\cdot ),
    \qquad\text{as $n\to\infty$,}
  \end{equation}
  in total variation norm.  Here define also
  \begin{equation}
    \label{Cdef}
    C(B) = \int_B c(x)\pi(dx), \qquad B \in {\cal B},
  \end{equation}
  and put $C=C({\cal X})$.
\item[(C2)] The pair $(X,\,P)$ is such that 
  \begin{equation}
    \label{slln}
    \lim_{n\to\infty}\frac{S_n}{n} = -a, \quad\text{ a.s., for some $a>0$}.
  \end{equation}
\end{itemize}

We need to recall the following definitions.  For any distribution
function~$H$ on $\R$, the integrated, or second-tail, distribution
$H^s$ is given by
\begin{displaymath}
  \overline{H}^s(y) = \min \left( 1,
    \int_y^{\infty} \overline{H}(t) dt
  \right).
\end{displaymath}
A distribution function $H$ on $\Rp$ is \emph{long-tailed} if
$\overline{H}(y)>0$ for all $y$ and, for any fixed $z>0$,
$$
\frac{\overline{H}(y+z)}{\overline{H}(y)} \to 1 \qquad\text{as
  $y\to\infty$.}
$$
A distribution function $H$ on $\Rp$ is
\emph{subexponential} if $\overline{H}(y)>0$ for all $y$ and
$$
\frac{\overline{H}^{*2}(y)}{\overline{H}(y)} \to 2 \qquad\text{as
  $y\to\infty$.}
$$
(where $H^{*2}$ is the convolution of $H$ with itself).
It is well known that any subexponential distribution is long-tailed.

We now have Theorems~\ref{thlb} and \ref{thsu} below.

\begin{thm} \label{thlb} 
  Assume that the conditions~(C1) and (C2) hold and that $F^s$ is
  long-tailed.  Then
  \begin{equation} \label{lowerth}
    \liminf_{y\to\infty}
    \frac{{\bf P} (M>y, X_{\mu (y)}\in B )}{\overline{F}^s(y)} \geq
    \frac{C(B)}{a}, \qquad\text{for all $B\in {\cal B}$}.
  \end{equation}
\end{thm}

%We now seek conditions under which the lower bound of
%Theorem~\ref{thlb} becomes sharp, i.e.\ under which the following
%conclusion holds:
%\begin{equation}
%  \label{prop}
%  \lim_{y\to\infty}
%  \frac{{\bf P} (M>y, X_{\mu (y)}\in B )}{\overline{F}^s(y)} =
%  \frac{C(B)}{a}, \qquad\text{for all $B\in {\cal B}$}.
%\end{equation}

%The first such condition is very simple.

\begin{thm} \label{thsu}
% Denote by $G(y)$ the right-continuous
%version of $\sup_x F_x(y)$.  
  Assume that the conditions~(C1) and (C2) hold, that $F^s$ is
  subexponential, and that there exists a distribution function $G$
  with negative mean
  \begin{equation}
    \label{su}
    m(G) \equiv \int_{-\infty}^\infty t\,dG(t) < 0
  \end{equation}
  such that $\overline{F}_x(y) \le \overline{G}(y)$ for all $x$ and
  $y$.  Then
\begin{equation}
  \label{prop}
  \lim_{y\to\infty}
  \frac{{\bf P} (M>y, X_{\mu (y)}\in B )}{\overline{F}^s(y)} =
  \frac{C(B)}{a}, \qquad\text{for all $B\in {\cal B}$}.
\end{equation}
\end{thm}

\begin{rem} 
  In an obvious sense, the best candidate for the distribution function
  $G$ in Theorem~\ref{thsu} is the right-continuous version of
  $\tilde{G}(y)=\sup_xF_x(y)$.
\end{rem}

%\sz{Take definition and examples of regenerative process (below) to
%  the Introduction?}

As mentioned above, the conditions~(C1) and (C2) are frequently
satisfied in applications.  Perhaps the most common instance occurs in
the case where $X$ is a \emph{regenerative process}.  Here there
exists an increasing integer-valued random sequence $0 \leq T_0 < T_1
< T_2 < \ldots $ a.s., such that, if $\tau_0=T_0$, $\tau_n =
T_n-T_{n-1}$, $n\ge1$, then
\begin{align*}
  Z_0 & = \{\tau_0 ; X_1, \ldots , X_{T_0} \},\\
%  \intertext{and}
  Z_n & = \{\tau_n ; X_{T_{n-1}+1}, \ldots , X_{T_n} \}, \qquad n\ge1,
\end{align*}
are mutually independent for $n\geq 0$ and identically distributed
for $n\geq 1$.  Here $\tau_0$ is the length of the $0$th cycle,
$\tau \equiv \tau_1$ the length of the first cycle, etc; in
particular, if $\tau_0=0$, then the $0$th cycle is empty.

Assume also, for $X$ regenerative as above, that ${\bf E}\tau$ is
finite.  Then the condition~(C1) is automatically satisfied with
\begin{displaymath}
  d=\text{GCD}\{n:\ {\bf P} (\tau = n)>0 \}
\end{displaymath}
and the probability measure $\pi$ given by
\begin{equation}
  \label{regpi}
  \pi (B) = {\bf E}
  \left(
    \sum_{i=T_0+1}^{T_1} {\bf I} (X_i \in B )
  \right),
  \qquad B\in {\cal B},
\end{equation}
where ${\bf I}$ is the indicator function.  Suppose also that the
modulated random walk~$S_n$ is constructed as above, that
\begin{equation}
  \label{absrdrift}
  \int_{\cal X}{\bf E}(|\xi^x_1|)\pi(dx) < \infty,
\end{equation}
and that
\begin{equation}
  \label{rdrift}
  \int_{\cal X}{\bf E}\xi^x_1\,\pi(dx) = -a, \quad\text{ for some $a>0$}.
\end{equation}
Then, it is an elementary exercise, using the Strong Law of Large
Numbers, to show that the condition~(C2) is satisfied with $a$ as
given by \eqref{rdrift}.

\begin{example}~\label{ex1}
  For a particular example of such a random walk modulated by a
  regenerative process, consider a stable tandem queue $GI/GI/1\to
  GI/1$ which is defined by three mutually independent sequences
  $\{t_n\}$, $\{ \sigma_n^{(1)} \}$, and $\{\sigma_n^{(2)} \}$ of
  i.i.d.\ random variables with ${\bf E}t_1 > \max \{ {\bf E}
  \sigma_1^{(1)}, {\bf E} \sigma^{(2)}_1 \}$.  Here the $t_n$ are the
  inter-arrival times at the first queue, while the $\sigma^{(1)}_n$
  and the $\sigma^{(2)}_n$ are the service times at the first and
  second queues respectively.  Let $\{X_n\}$ be the sequence of
  inter-departure times from the first queue.  Then this sequence is
  regenerative (the regeneration indices corresponding to those
  customers who arrive to find the first queue empty), and the
  distribution of $X_n$ converges in the total variation norm to a
  stationary distribution with mean ${\bf E} t_1$.  Consider the
  sequence $\xi_n^{X_n} = \sigma_n^{(2)} - X_n$.  Under natural
  conditions (see Theorems~\ref{thexp}--\ref{thw} below), we can show
  that the tail distribution of the supremum of a modulated random
  walk with increments $\xi_n^{X_n}$ asymptotically coincides with
  that of the stationary waiting time in the second queue.
\end{example}

In the case of a regenerative process as above we obtain, in
Theorems~\ref{thexp} and \ref{thup} below, the conclusion~\eqref{prop}
of Theorem~\ref{thsu} under weaker conditions than that given by
\eqref{su}.  In each case the cost is that of suitable conditions on
the distributions of the cycle times~$\tau_0$ and $\tau$.
Both the theorems are adapted to typical queueing theory applications.

\begin{thm} \label{thexp} 
  Assume that $X$ is regenerative with ${\bf E}\tau<\infty$, that
  the conditions~\eqref{absrdrift} and \eqref{rdrift} hold (with $\pi$
  as given by \eqref{regpi}), and that $F^s$ is subexponential.
  Assume also that
  \begin{equation} \label{tautail} 
    {\bf P} (b\tau_0 > y)  = o(\overline{F}^s(y)),
    \quad
    {\bf P} (b\tau > y)  = o(\overline{F}(y)),
    \quad\text{as $y\to\infty$},
  \end{equation}
  for all $b>0$.
  Then the conclusion~\eqref{prop} of Theorem~\ref{thsu} again follows.
%  \sz{Restate this here?}
\end{thm}

\begin{rem}
  As already discussed, the assumptions of Theorem~\ref{thexp} ensure
  that the earlier conditions~(C1) and (C2) are satisfied.
\end{rem}

\begin{rem} 
  It will follow from the proof of Theorem \ref{thexp} that it is
  enough to assume that the condition \eqref{tautail} holds for a
  certain sufficiently large $b$.
\end{rem}

\begin{rem}
  The condition \eqref{tautail} holds for all $b>0$ if there exists
  some $\lambda>0$ such that both ${\bf E} \exp (\lambda \tau )$ and
  ${\bf E} \exp (\lambda \tau_0)$ are finite.
\end{rem}

Under the conditions of Theorem~\ref{thup} we relax the requirement
that the condition~\eqref{tautail} hold for all $b>0$.  This
requirement may fail to be satisfied in some examples where the
conditions of Theorem~\ref{thup} are, however, quite natural---see,
e.g., Example~\ref{ex1}.

%\sz{Reference to our queueing example required here!}

\begin{thm} \label{thup} 
  Assume that $X$ is regenerative with ${\bf E}\tau<\infty$, that
  the conditions~\eqref{absrdrift} and \eqref{rdrift} hold, and that
  $F^s$ is subexponential.  Assume also that there exists a family
  $\{G_x\}_{x\in\cal X}$ of distribution functions on $\Rp$ such that
  $G_x(y)$ is measurable in $x$ for all $y$,
  \begin{equation} \label{FG} 
    \overline{F}_x(y) \le \overline{G}_x(y) \quad
    \text{ for all $x$ and for all $y$}, 
  \end{equation}
  and, for each $x$, $G_x$ may be represented as a distribution
  function of a difference of two independent r.v.s
  \begin{equation} \label{dif}
    G_x(y) = {\bf P} (\zeta - b^x \leq y),
  \end{equation}
  where the distribution of $\zeta$ does not depend on $x$,
  \begin{equation}\label{zbd}
    \limsup_{y\to\infty}\frac{{\bf P}(\zeta>y)}{\overline{F}(y)} < \infty,
  \end{equation}
  $b^x$ is non-negative a.s., and
  \begin{equation} \label{bb}
    {\bf E}_{\pi} b^X \equiv
    \int {\bf E} b^x\,\pi (dx) >  {\bf E} \zeta.
  \end{equation}
  Assume further that the condition \eqref{tautail} holds for some
  $b>{\bf E} \zeta$.  Then the conclusion~\eqref{prop} of
  Theorem~\ref{thsu} follows once more.
%  \sz{Restate this here?}
\end{thm}

\begin{rem}
  The assumptions on $X$ in Theorem~\ref{thup} again ensure that the
  earlier condition~(C1) is satisfied, while it follows from
  \eqref{FG}, \eqref{dif} and \eqref{bb} that the
  condition~\eqref{rdrift}, and so the condition~(C2), is satisfied.
% Note also that our global
%  assumption~\eqref{bounded} is in fact guaranteed by \eqref{FG}
%and \eqref{dif}
%  (since the random variable $b^x$ is nonnegative).
\end{rem}

\begin{rem}
  It is easy to see that the asymptotics for ${\bf P} (M>y)$ may be
  quite different from \eqref{prop} if the assumption \eqref{tautail}
  fails.  For instance, assume that the remaining conditions of
  Theorem~\ref{thup} hold with ${\cal X} = {\cal R}$, $F_x=G_x$ for
  all~$x$, $\zeta \geq 1$ a.s., and $b^x\equiv0$ for all $x\neq 0$.
  Assume also that $X_0 = 0$, $T_0=0$, $\tau \equiv \tau_1 = \min \{
  n>0\colon X_n = 0 \}$.  The condition~\eqref{bb} here becomes ${\bf
    E} b^0 > {\bf E} \zeta {\bf E} \tau $.  Then
  $$
  M \geq \max (\tau_1 -1, \tau_1 - b_{\tau_1}^0 + \tau_2 -1, \ldots )
  \equiv M^{*}.
  $$
  If the second tail for ${\bf P} (\tau >t)$ is subexponential (and so
  also long-tailed), then,
  by \eqref{11},
  $$
  {\bf P} (M^{*}>y) \sim
  \frac{1}{{\bf E} (b^0-\tau)} \int_y^{\infty}
  {\bf P} (\tau >t) dt .
  $$ 
\end{rem}

Finally, again in the case where $X$ is regenerative, we consider the
process~$\{W_n\}_{n\ge0}$ defined earlier by \eqref{wrec}.  In the
special case where $F_x=F$ for all $x\in\cal X$ (so that
$\{\xi_n\}$ is an i.i.d.  sequence), it is well known that the
distributions of $W_n$ and $S_n$ coincide.  However, this is not
generally the case in the present setting.

\begin{thm} \label{thw}
  Assume that $X$ is regenerative with ${\bf E}\tau<\infty$.
  Then, under the conditions of either Theorem \ref{thsu},
  Theorem~\ref{thexp} or Theorem~\ref{thup},
  \begin{equation} \label{w}
    \lim_{y\to\infty}
    \frac{1}{\overline{F}^s(y)}
    \limsup_{n\to\infty}
    {\bf P} (W_n>y) =
    \lim_{y\to\infty}
    \frac{1}{\overline{F}^s(y)}
    \liminf_{n\to\infty}
    {\bf P} (W_n>y) = 
    \frac{C}{a}.
  \end{equation}
\end{thm}

\begin{rem}
  In fact, under the conditions of Theorem~\ref{thexp} or \ref{thup},
  the condition on $\tau_0$ (in \eqref{tautail}) is not required for
  Theorem~\ref{thw}.
\end{rem}

\section{Useful Properties}
\label{sec:properties}

We recall some known properties of distributions.  For any
distribution function~$G$ on $\R$ let 
\begin{displaymath}
      m(G) \equiv \int_{-\infty}^\infty t\,dG(t) %< 0
\end{displaymath}
denote its mean.  Further, we make the convention that, for
distribution functions $G$ and $H$, we write
$\overline{H}(y)\sim0\cdot\overline{G}(y)$ if
$\overline{H}(y)=o(\overline{G}(y))$ as $y\to\infty$.

\begin{property}\label{p:ext}
  Suppose that distribution functions $G$ and $H$ are such that $m(G)$
  is finite, $m(H)=-h$ for some $h>0$, and
  $\overline{H}(y)=o(\overline{G}(y))$ as $y\to\infty$.  Then, for any
  $\varepsilon>0$ we can find a distribution function~$H_\varepsilon$
  such that $\overline{H}(y)\le\overline{H}_\varepsilon(y)$ for all
  $y$, $m(H_\varepsilon)\le-h/2$ and
  $\overline{H}_\varepsilon(y)=\varepsilon\overline{G}(y)$ for all
  sufficiently large $y$.
\end{property}

\begin{property}\label{p:com}
  Suppose that distribution functions $G$ and $H$ are such that $G^s$
  exists, and that $\overline{H}(y)\sim c\overline{G}(y)$ as
  $y\to\infty$ for some $c\ge0$.  Then $H^s$ exists and
  $\overline{H}^s(y)\sim c\overline{G}^s(y)$ as $y\to\infty$.
\end{property}

\begin{property}
  Suppose that a distribution function $G$ is such that its second
  tail distribution $G^s$ is long-tailed. Then
  \begin{equation}
    \label{lt1}
    \overline{G}(y) = o(\overline{G}^s(y)) \quad
    \mbox{as} \quad y\to\infty.
  \end{equation}
  Further, for any $g>0$ and any sequence $\{\alpha_n\}$ such that
  $\alpha_n\to\alpha$ as $n\to\infty$,
  \begin{equation}
    \label{lt2}
    \lim_{y\to\infty}\frac{1}{\overline{G}^s(y)}
    \sum_{n=k}^\infty \alpha_n \overline{G}(y+l+ng) = \frac{\alpha}{g}
    \qquad\text{for all $k$ and for all $l$}.
  \end{equation}
\end{property}

\begin{property}
  Suppose that distribution functions $G$ and $H$ are such that
  $\overline{H}(y)\sim c\overline{G}(y)$ as $y\to\infty$ for some
  $c>0$.  Then if $G$ is subexponential, $H$ is subexponential, while
  if $G^s$ subexponential, $H^s$ is subexponential and
  $\overline{H}^s(y)\sim c\overline{G}^s(y)$ as $y\to\infty$. 
\end{property}

\begin{property}\label{p:conv}
  Let $\xi_1,\xi_2,\ldots ,\xi_n$ be $n$ mutually independent r.v.s
  and $G$ a subexponential distribution such that, for
  $i=1,2,\ldots,n$, ${\bf P} (\xi_i >y) \sim c_i\overline{G}(y)$ as
  $y\to \infty$, where $c_1,c_2,\ldots ,c_n\geq 0$.  Then
  \begin{displaymath}
    {\bf P} (\xi_1+\xi_2+\ldots +\xi_n>y) \sim (c_1+c_2+\ldots+c_n)
    \overline{G}(y) \qquad\text{as $y\to\infty$.}
  \end{displaymath}
\end{property}

\begin{property}
  Let $\{\xi_n\}_{n\ge1}$ be an i.i.d.\ sequence of nonnegative random
  variables with subexponential distribution $G$.  For any $n$, put
  $$
  \alpha_n = \sup_{y\geq 0} \frac{{\bf P} (\xi_1+\ldots +
    \xi_n>y)}{{\bf P} (\xi_1>y)} \equiv \sup_{y\geq 0}
  \frac{\overline{G}^{*n}(y)}{\overline{G}(y)}.
  $$
  Then, for any $u>0$ one can choose $k>0$ such that $\alpha_n \leq
  k(1+u )^n$ for all $n$.
\end{property}

\begin{property}[Veraverbeke's Theorem]
  Let $\{\xi_n\}_{n\ge1}$ be an i.i.d.\ sequence of random variables
  with distribution function $G$ and a negative mean $-g = {\bf E}
  \xi_1 <0$.  Suppose that the second-tail distribution $G^s$ is
  subexponential. Set $S'_n=\sum_{i=1}^n \xi_i$, and
  $M'=\max(0,\sup_nS'_n)$.  Then, as $y\to\infty$,
  \begin{equation} \label{Verav}
    {\bf P}(M' > y) \sim {\bf P} \left( \bigcup_{n\ge1} \{ \xi_n > y+ng \}
    \right) \sim \sum_{n\ge1} {\bf P} (\xi_n > y+ng) \sim \frac{1}{g}
    \overline{G}^s(y).
  \end{equation}
\end{property}

Thus, under the conditions of Veraverbeke's Theorem, the supremum~$M'$
is large if and only if one of summands is large.  The following three
properties are all corollaries of Veraverbeke's Theorem.  In
particular Property~\ref{p:vc2} follows easily on using also
Property~\ref{p:ext} above.

\begin{property} \label{p:vc1}
  Under the conditions of Veraverbeke's Theorem above, for any
  $\tilde{g}\in (0,g)$,
  $$
  \sum_{n=1}^{\infty} 
  {\bf P} (M'_n \leq y, S'_n \in (-n\tilde{g}, y], S'_{n+1}>y)
  = o(\overline{G}^s(y))
  \qquad\text{as $y\to\infty$},
  $$
  where, for each $n$,  $M'_n=\max(0,\max_{1\le{}i\le{}n}S'_i)$.
\end{property} 

\begin{property} \label{p:vc2}
  Let $\{\xi_n\}_{n\ge1}$ be an i.i.d.\ sequence of random variables
  with distribution function $H$ and negative mean~${\bf E}\xi_1<0$.
  Suppose that $\overline{H}(y)=o(\overline{G}(y))$, as $y\to\infty$,
  for some distribution function~$G$ whose second-tail distribution
  $G^s$ is subexponential.  Set $S'_n=\sum_{i=1}^n \xi_i$ and
  $M'=\max(0,\sup_n S'_n)$.  Then
  \begin{displaymath}
    {\bf P}(M' > y) = o(\overline{G}^s(y)) \qquad\text{as $y\to\infty$.}
  \end{displaymath}
\end{property}

\begin{property}
  Let $\{\xi_n\}_{n\ge1}$ be an i.i.d.\ sequence of random variables
  with distribution function $H$ and negative mean~${\bf E}\xi_1<0$.
  Suppose that $\overline{H}(y)\sim c(\overline{G}(y))$, as
  $y\to\infty$, for some $c\geq 0$ and some distribution function~$G$
  whose second-tail distribution $G^s$ is subexponential.  Let $\tau'$
  be an independent positive integer-valued random variable.  Then
  $$
  {\bf P} \left(\max_{1\leq n\leq \tau'} \sum_{i=1}^n\xi_i >
    y\right) = o(\overline{G}^s(y)) \qquad\text{as $y\to\infty$.}
  $$
\end{property}

\section{Proofs}
\label{sec:proofs}

\begin{proof}[Proof of Theorem~\ref{thlb}]
  We prove the theorem in the case where the constant~$d$ of the
  condition~(C1) is equal to $1$.  The modification required for the
  general case is quite obvious.  By the Strong Law of Large Numbers,
  for any $\varepsilon \in (0,a)$, we can choose $R\equiv
  R(\varepsilon )$ such that
  \begin{displaymath}
    {\bf P}
    \left(S_n \in
      \left[-R - n(a+\varepsilon), R - n(a-\varepsilon)\right] 
      \quad
      \text{for all $n=0,1,2,\ldots$}
    \right)
    \geq 1-\varepsilon.
  \end{displaymath}
  Put
  $$
  D_n =
  \left\{
    S_i \in
    \left[-R - i(a+\varepsilon), R - i(a-\varepsilon)\right] 
    \quad
    \text{for all $i=0,1,2, \ldots ,n$}
  \right\}
  $$
  and $D\equiv D_{\infty}$. Since $D_{\infty}\subseteq D_n$ for
  all $n$, ${\bf P} (D_n)\geq 1-\varepsilon $.
  
  Now, for all sufficiently large $y>0$,
  \begin{align*}
    & {\bf P} (M > y, X_{\mu (y)} \in B)\nonumber\\
    & = \sum_{n=0}^\infty {\bf P}(M_n \le y, S_{n+1} > y, X_{n+1}
    \in B)\nonumber\\
    & \ge \sum_{n=0}^\infty {\bf P}(D_n, S_{n+1} > y, X_{n+1}
    \in B)\nonumber\\ 
    & \ge \sum_{n=0}^\infty \int_B {\bf P}(D_n, X_{n+1}\in dx)
    \overline{F}_x(y + R + n(a+\varepsilon))\nonumber\\ 
    & \ge \sum_{n=0}^\infty
    \left[\int_B {\bf P}(X_{n+1}\in dx)
      \overline{F}_x(y + R + n(a+\varepsilon))
      - {\bf P}(\overline{D})L\overline{F}(y + R + n(a+\varepsilon))
    \right]\nonumber\\
    & \ge \sum_{n=0}^\infty
    \left[\int_B \pi(dx)
      \overline{F}_x(y + R + n(a+\varepsilon))
      - \left({\bf P}(\overline{D})+\delta_{n+1}\right)
      L\overline{F}(y + R + n(a+\varepsilon))
    \right],\label{pause}
  \end{align*}
  where, for each $n$, $\delta_n=\sup_B|{\bf P}(X_n\in B) - \pi(B)|$
  is the distance in total variation between the distributions of
  $X_n$ and $\pi$.  The condition~(C1) implies that $\delta_n\to0$ as
  $n\to\infty$.  Since $F^s$ is long-tailed, it now follows from
  \eqref{tail}, \eqref{comp} and \eqref{lt2} that, for all $x$,
  \begin{equation*}\label{ly1}
    \lim_{y\to\infty}
    \frac{1}{\overline{F}^s(y)}
    \sum_{n=0}^\infty
    \overline{F}_x(y + R + n(a+\varepsilon))
    = \frac{c(x)}{a+\varepsilon}
  \end{equation*}
  and that
  \begin{align}
    \limsup_{y\to\infty}
    \frac{1}{\overline{F}^s(y)}
    \sum_{n=0}^\infty
    \overline{F}_x(y + R + n(a+\varepsilon))
    & \le
    \lim_{y\to\infty}
    \frac{L}{\overline{F}^s(y)}
    \sum_{n=0}^\infty
    \overline{F}(y + R + n(a+\varepsilon))\nonumber\\
    & = \frac{L}{a+\varepsilon} \label{lbct}
  \end{align}
  (where the convergence to the limit above is of course independent
  of $x$).  Hence, by the Bounded Convergence Theorem,
  \begin{equation}
    \label{bct}
    \lim_{y\to\infty}
    \frac{1}{\overline{F}^s(y)}
    \sum_{n=0}^\infty
    \int_B \pi(dx)
    \overline{F}_x(y + R + n(a+\varepsilon))
    = \frac{C(B)}{a+\varepsilon}.
  \end{equation}
  Also, again from \eqref{lt2},
  \begin{displaymath}
    \lim_{y\to\infty}
    \frac{1}{\overline{F}^s(y)}
    \sum_{n=0}^\infty
    \delta_{n+1}\overline{F}(y + R + n(a+\varepsilon))
    = 0.
  \end{displaymath}
  Thus, again using \eqref{lbct},
  \begin{displaymath}
    \liminf_{y\to\infty}
    \frac{1}{\overline{F}^s(y)}
    {\bf P} (M > y, X_{\mu (y)} \in B)
    \ge \frac{C(B)-L{\bf P}(\overline{D})}{a+\varepsilon}.
  \end{displaymath}
    Now let $\varepsilon\to 0$ to obtain the required result.
\end{proof}

We now give two lemmas which are required for the remaining results.

\begin{lem}\label{lem1}
  Suppose that the conditions of Theorem~\ref{thlb} hold and that
  \begin{equation} \label{prop2}
    \limsup_{y\to\infty} 
    \frac{{\bf P} (M>y)}{\overline{F}^s (y)} \leq \frac{C}{a}.
  \end{equation}
  Then the conclusion~\eqref{prop} follows.
\end{lem}
\begin{proof}
From \eqref{prop2}, for any $B\in {\cal X}$,
\begin{align*}
  \frac{C}{a} & \geq \limsup_{y\to\infty} 
  \frac{{\bf P} (M>y)}{\overline{F}^s (y)} \\
  & = \limsup_{y\to\infty} \left(
    \frac{{\bf P} (M>y, X_{\mu (y)} \in B)}{\overline{F}^s (y)}
    + 
    \frac{{\bf P} (M>y, X_{\mu (y)} \in 
      \overline{B})}{\overline{F}^s (y)}
  \right)\\
  & \geq \limsup_{y\to\infty} 
  \frac{{\bf P} (M>y, X_{\mu (y)} \in B)}{\overline{F}^s (y)}
  +
  \liminf_{y\to\infty}
  \frac{{\bf P} (M>y, X_{\mu (y)} \in 
    \overline{B})}{\overline{F}^s (y)} \\
  & \geq \limsup_{y\to\infty}
  \frac{{\bf P} (M>y, X_{\mu (y)} \in B)}{\overline{F}^s (y)} +
  \frac{C(\overline{B})}{a},
\end{align*}
where the last inequality follows by Theorem~\ref{thlb}.
Since $C=C(B)+ C(\overline{B})$, the conclusion \eqref{prop} follows
as required.
\end{proof}

In each of the proofs of Theorems~\ref{thsu}, \ref{thexp} and
\ref{thup} we show that, for all $\varepsilon$ satisfying
$0<\varepsilon<a$, there exists $R>0$ (depending on $\varepsilon$)
such that, if, for each $n=1,2,\ldots$,
\begin{equation}
  \label{ddef}
  D'_n \equiv
  \{ S_j \leq R-j(a-\varepsilon )
  \quad \text{for all $j=1,\ldots , n-1$};
  \ S_{n+i}-S_n \leq R
  \quad \text{for all $i=1,2, \ldots$} \},
\end{equation}
then ${\bf P}(D'_n)>1-\varepsilon$ for all $n$.  In each case we then
require Lemma~\ref{lem2} below to complete the proof.

\begin{lem}\label{lem2}
  Suppose that $F^s$ is subexponential, that there exist a sequence of
  i.i.d.\ random variables~$\{\psi_n\}_{n\ge1}$ and a constant~$L_1$
  such that ${\bf{}E}\psi_1<0$ and
  \begin{equation}\label{l1bdd}
    {\bf{}P}(\psi_1>y)\le{}L_1\overline{F}(y)
  \end{equation}
  for all $y\ge0$, and that $\psi_n$ is independent of $D'_n$ for all
  $n\ge1$.  Suppose further that the condition~(C1) is satisfied and
  that
  \begin{equation}
    \label{mmpsi}
    {\bf P}(M>y) \le {\bf P}(M > y, M^\psi > y) + o(\overline{F}^s(y))
    \quad\text{ as $y\to\infty$},
  \end{equation} 
  where $M^\psi=\max(0,\sup_n\sum_{i=1}^n\psi_i)$.  Then the
  conclusion~\eqref{prop} follows.
\end{lem}

\begin{proof}
  As in the proof of Theorem~\ref{thlb}, we assume that the
  constant~$d$ of the condition~(C1) is equal to $1$.  We may further
  assume, without loss of generality, that the condition~\eqref{l1bdd}
  is satisfied with equality for all sufficiently large $y$.  (If this
  is not the case we can use Property~\ref{p:ext} of
  Section~\ref{sec:properties} to replace $\{\psi_n\}_{n\ge1}$ with
  i.i.d.\ sequence $\{\tilde{\psi}_n\}_{n\ge1}$ satisfying all the
  conditions of the lemma and with also the required equality in
  \eqref{l1bdd}.)  It follows that the common distribution of the
  random variables $\psi_n$ has a second-tail distribution which is
  subexponential.  Thus, if $g=-{\bf{}E}(\psi_1)$ (so $g>0$), it
  follows from the conditions of the lemma and Veraverbeke's Theorem
  that
  \begin{align}
    {\bf P}(M>y)
    & \le {\bf P}(M > y, M^\psi > y) + o(\overline{F}^s(y))\nonumber\\
    & = \sum_{n=1}^{\infty} {\bf P} (M>y,\psi_n>y+ng) +
    o(\overline{F}^s(y))\nonumber\\
    & \le \Sigma_1 + \Sigma_2 + o(\overline{F}^s(y)),\label{pause1}
  \end{align}
  where 
  \begin{displaymath}
    \Sigma_1 = \sum_{n=1}^{\infty} {\bf P}(D'_n,M>y),\qquad
    \Sigma_2 = \sum_{n=1}^{\infty}
    {\bf P}(\overline{D}'_n,M>y,\psi_n>y+ng). 
  \end{displaymath}
  Since, for each $n$, $\psi_n$ is independent of $D'_n$, we have,
  using \eqref{lt2},
  \begin{equation}\label{pause2}
    \Sigma_2 \le \sum_n {\bf P} (\overline{D}'_n)
    {\bf P} (\psi_n > y+ng)
    \leq (1+o(1))\frac{\varepsilon L_1}{g}\overline{F}^s(y)
    \quad\text{as $y\to\infty$}.
  \end{equation}
  We now consider $\Sigma_1$. Take $y>R$. For any $n$, the event
  \begin{equation} \label{empty}
    V_n \equiv D'_n \cap 
    \{ \xi_n^{X_n} \leq y-2R + (n-1)(a-\varepsilon )\} \subseteq
    \{ M \leq y \}.
  \end{equation}
  To see this, note that, on the set~$V_n$,
  $S_j \leq R-j(a-\varepsilon )$ for all $j<n$,
  $$
  S_n = S_{n-1}+\xi_n^{X_n} \leq y-R
  $$
  and, for $i=1,2,\ldots$,
  $$
  S_{n+i} = S_n + (S_{n+i}-S_n) \leq y-R+R = y.
  $$
%  \sz{Do we need the above explanation?}

  Thus, from \eqref{empty},
  \begin{align*}
    \Sigma_1
    & \leq \sum_n {\bf P}(\xi_n^{X_n} > y-2R+(n-1)(a-\varepsilon)) \\
    & = \sum_n \int_{\cal X} {\bf P} (X_n\in dx)
    \overline{F}_x(y-2R+(n-1)(a-\varepsilon)) \\
    & \leq \sum_n \int_{\cal X} \pi (dx) 
    \overline{F}_x(y-2R+(n-1)(a-\varepsilon)) 
    + L\sum_n \delta_n \overline{F}(y-2R+ (n-1) (a-\varepsilon )),
  \end{align*}
  where, as in the the proof of Theorem~\ref{thlb}, $\delta_n$ is the
  distance in total variation between the distributions of $X_n$ and
  $\pi$, and so tends to $0$ as $n\to\infty$.  Exactly as in that
  proof, it now follows from \eqref{lt2} and the Bounded Convergence
  Theorem that
  \begin{displaymath}
    \limsup_{y\to\infty}\frac{\Sigma_1}{\overline{F}^s(y)}
    \le \frac{C}{a-\varepsilon}.
  \end{displaymath}
  It now follows, on recalling \eqref{pause1} and
  \eqref{pause2} and letting $\varepsilon\to0$, that the
  condition~\eqref{prop2} of Lemma~\ref{lem1} is satisfied.  The
  required conclusion~\eqref{prop} now follows from that lemma.
\end{proof}

\begin{proof}[Proof of Theorem~\ref{thsu}]
  
  It follows from the condition~\eqref{comp} that, without loss of
  generality, we can assume that $\overline{G}(y) \leq
  L\overline{F}(y)$ for all~$y$.   
  Let $\{\alpha_n\}_{n\ge1}$ be an i.i.d.\ sequence of random
  variables uniformly distributed on $(0,1)$ and independent of
  $X=\{X_n\}$.  Construct the required family of random
  variables~$\{\xi_n^x\}_{n\ge1}$ by defining, for each $n$,
  $\xi_n^x=F^{-1}_x(\alpha_n)$; for each $n$ define also
  $\psi_n=G^{-1}(\alpha_n)$.  Here, for any distribution function~$H$,
  the quantile function~$H^{-1}$ is given by
  $$
  H^{-1}(t) = \sup \{ z: \ H(z) \leq t \}.
  $$
  Note that the pairs~$(\xi_n^x,\psi_n)$, $n\ge1$, are independent
  in $n$, that the sequence $\{\psi_n\}_{n\ge1}$ is i.i.d.\ with
  ${\bf{}E}\psi_1<0$ (from \eqref{su}) and distribution function~$G$,
  and that
  \begin{equation}
    \xi_n^x\leq\psi_n \quad\text{a.s..} \label{xibdd}
  \end{equation}
  Put $S^{\psi}_n =\sum_{j=1}^n \psi_j$ and
  $$
  M^{\psi} = \max(0,\sup_n S^{\psi}_n).
  $$
  From the SLLN for $\{\psi_n\}$ and from the condition (C2), 
  for any $\varepsilon >0$, there exists $R>0$ such that, for any
  $n=1,2,\ldots$,
  \begin{displaymath}
    {\bf P} (S_j < R-j(a-\varepsilon),\ j=1,2,\ldots,n-1;\ %
    S^{\psi}_{n+i} - S^{\psi}_n < R, \ i=1,2,\ldots ) > 1-\varepsilon.
  \end{displaymath}
  Hence, from \eqref{xibdd}, ${\bf P}(D'_n)>1-\varepsilon$ for all
  $n$, where each $D'_n$ is as given by \eqref{ddef}.  Also from
  \eqref{xibdd},
  \begin{displaymath}
    {\bf P} (M>y) = {\bf P}(M>y, M^{\psi} > y).
  \end{displaymath}
  It is now easy to check that all the conditions of Lemma~\ref{lem2}
  are satisfied, with each $\psi_n$ and $D'_n$ as given here, and the
  required result now follows from that lemma.
\end{proof}

The following further two lemmas are also required in each of the
proofs of Theorems~\ref{thexp} and \ref{thup} (where in each case $X$
is regenerative).

\begin{lem}\label{lem3}
  Suppose that $X$ is regenerative with ${\bf E}\tau<\infty$ and also
  that $F^s$ is subexponential.  Let $\{\{\eta^x_n\}_{x\in\cal
    X}\}_{n\ge1}$ be a sequence of families of random variables such
  that these families are independent and identically distributed in
  $n$
  and are further independent of $X$.  Suppose further that there
  exists a constant~$b>0$ satisfying the condition~\eqref{tautail} and
  such that
  \begin{equation}
    \label{hbd}
    \eta^x_1 \le b \quad\text{a.s. \qquad for all $x$},
  \end{equation}
  and that
  \begin{equation}
    \label{hdrift}
    \int_{\cal X} {\bf E}\eta^x_1\,\pi(dx) < 0.
  \end{equation}
  Define
  \begin{displaymath}
    M^\eta = \max\left(0,\sup_n \sum_{i=1}^n \eta^{X_i}_i\right).
  \end{displaymath}
  Then
  \begin{displaymath}
    {\bf P}(M^\eta>y) = o(\overline{F}^s(y))
    \qquad\text{as $y\to\infty$}.
  \end{displaymath}
\end{lem}

\begin{proof}
  Define
  \begin{displaymath}
    \beta_n = \sum_{i=T_{n-1}+1}^{T_n}\eta^{X_i}_i, \qquad n\ge 1.
  \end{displaymath}
  Observe that $\{\beta_n\}_{n\ge1}$ is an i.i.d.\ sequence with, from
  \eqref{hbd} and \eqref{hdrift},
  \begin{displaymath}
    {\bf E}\beta_1 < 0, \qquad\qquad
    \beta_n \le b\tau_n, \quad n \ge 1.
  \end{displaymath}
  Since also $X$ is regenerative with ${\bf E}(\tau)<\infty$, we can
  choose $K>0$ sufficiently large that if
  \begin{equation}
    \label{gammadef}
    \gamma_n = \max(\beta_n, b\tau_n - K), \qquad n \ge 1,
  \end{equation}
  then $\{\gamma_n\}_{n\ge1}$ is an i.i.d.\ sequence with
  \begin{equation}
    \label{gammaprop}
    {\bf E}\gamma_1 < 0, \qquad\qquad
    \gamma_n \le b\tau_n, \quad n \ge 1.
  \end{equation}
  Define also
  \begin{displaymath}
    M^\gamma = \max\left(0, \sup_{n\ge1}\sum_{i=1}^n \gamma_i\right).
  \end{displaymath}
  It follows from \eqref{gammaprop}, the assumed
  condition~\eqref{tautail} (for $b$ as given) and the extension of
  Veraverbeke's Theorem given by Property~\ref{p:vc2} of
  Section~\ref{sec:properties}, that
  \begin{equation}
    \label{mgtail}
    {\bf P}(M^\gamma>y) = o(\overline{F}^s(y)),
    \quad\text{as $y\to\infty$}.
  \end{equation}
  Now
  \begin{align*}
    M^\eta
    & \le b\tau_0
    + \sup(b\tau_1, \beta_1+b\tau_2, \beta_1+\beta_2+b\tau_3,\ldots)\\
    & \le b\tau_0 + K
    + \sup(\gamma_1, \gamma_1+\gamma_2, \gamma_1+\gamma_2+\gamma_3,\ldots)\\
    & \le  b\tau_0 + K + M^\gamma,
  \end{align*}
  where the second inequality above follows from \eqref{gammadef}.
  Further $\tau_0$ and $M^\gamma$ are independent.  The required
  result now follows from \eqref{mgtail}, the assumed
  condition~\eqref{tautail} and Property~\ref{p:conv} of
  Section~\ref{sec:properties}.
\end{proof}
 
The following lemma combines the results of Lemmas~\ref{lem2} and
\ref{lem3} to provide a set of conditions for the regenerative case
under which there follows the desired conclusion~\eqref{prop} of both
Theorems~\ref{thexp} and \ref{thup}.

\begin{lem}\label{lem4}
  Suppose that $X$ is regenerative with ${\bf E}\tau<\infty$ and also
  that $F^s$ is subexponential.  Suppose also that there exist a
  sequence of i.i.d.\ random variables~$\{\psi_n\}_{n\ge1}$ and a
  constant~$L_1$ satisfying the conditions of Lemma~\ref{lem2}, i.e.\ %
  that
  \begin{equation}\label{psicond1}
    {\bf{}E}(\psi_1)<0, \qquad
    {\bf{}P}(\psi_1>y)\le{}L_1\overline{F}(y) \quad\text{for all $y\ge0$,}
  \end{equation}
  and that
  \begin{equation}
    \label{psicond2}
    \text{$\psi_n$ is independent of $D'_n$ for all $n\ge1$}
  \end{equation}
  (where $D'_n$ is as given by \eqref{ddef}).  Suppose further that
  there exists a sequence of families of random
  variables~$\{\{\eta^x_n\}_{x\in\cal{}X}\}_{n\ge1}$ and a
  constant~$b>0$ satisfying all the conditions of Lemma~\ref{lem3},
  and that
  \begin{equation}
    \label{xipsieta}
    \xi^x_n \le \psi_n + \eta^x_n, \qquad x\in{\cal X}, \quad n \ge 1.
  \end{equation}
  Again define
  \begin{displaymath}
    M^\psi = \max\left(0,\sup_n \sum_{i=1}^n \psi_i\right), \qquad
    M^\eta = \max\left(0,\sup_n \sum_{i=1}^n \eta^{X_i}_i\right).
  \end{displaymath}
  Finally, suppose that $M^\psi$ and $M^\eta$ are independent.  Then
  the conclusion~\eqref{prop} follows.
\end{lem}

\begin{proof}
  As in the proof of Lemma~\ref{lem2} we may assume, without loss of
  generality, that $ {\bf{}P}(\psi_1>y)=L_1\overline{F}(y)$ for all
  sufficiently large $y$.  It then follows from the conditions on the
  sequence~$\{\psi_n\}$ and Veraverbeke's Theorem that
  \begin{equation}
    \label{mpsitail}
    {\bf P}(M^\psi>y) \sim L_1\overline{F}^s(y),
    \quad\text{as $y\to\infty$},
  \end{equation}
  while it follows from Lemma~\ref{lem3} that
  \begin{equation}\label{metatail}
    {\bf P}(M^\eta>y) = o(\overline{F}^s(y)),
    \quad\text{as $y\to\infty$}.
  \end{equation}
  From the condition~\eqref{xipsieta} we have that
  \begin{equation}
    \label{mxibdd}
    M \le M^\psi + M^\eta.
  \end{equation}
  Since also $M^\psi$ and $M^\eta$ are independent, it now follows
  from \eqref{mpsitail}, \eqref{metatail}, \eqref{mxibdd} and
  Property~\ref{p:conv} of Section~\ref{sec:properties} that
  \begin{equation}
    \label{mmpsi2}
    {\bf P}(M>y) = {\bf P}(M > y, M^\psi > y) + o(\overline{F}^s(y))
    \quad\text{ as $y\to\infty$}.
  \end{equation} 
  Finally, since $X$ is regenerative, the condition~(C1), and so now
  all the conditions of Lemma~\ref{lem2}, are satisfied and so the
  required conclusion~\eqref{prop} again follows from that lemma.
\end{proof}

\begin{proof}[Proof of Theorem~\ref{thexp}]
  We construct the sequences~$\{\psi_n\}_{n\ge1}$ and
  $\{\{\eta^x_n\}_{x\in\cal{}X}\}_{n\ge1}$ and the constants~$L_1$ and
  $b$ such that all the conditions of Lemma~\ref{lem4} are satisfied.
  
  It follows from \eqref{comp} that we can find a distribution
  function~$G$ on $\R$ such that
  \begin{equation}\label{sw}
    \overline{F}_x(y) \le \overline{G}(y)\le L\overline{F}(y),
  \end{equation}
  for all~$y$ and for all $x\in\cal X$.  As in the proof of
  Theorem~\ref{thsu}, let $\{\alpha_n\}_{n\ge1}$ be an i.i.d.\ %
  sequence of random variables uniformly distributed on $(0,1)$ and
  independent of $X=\{X_n\}$.  Again construct the required family of
  random variables~$\{\xi_n^x\}_{n\ge1}$ by defining, for each $n$,
  $\xi_n^x=F^{-1}_x(\alpha_n)$; for each $n$ define also
  $\zeta_n=G^{-1}(\alpha_n)$.  Then the pairs~$(\xi_n^x,\zeta_n)$,
  $n\ge1$, are independent in $n$, the sequence $\{\zeta_n\}_{n\ge1}$
  is i.i.d., and
  \begin{equation}
    \label{xibdd2}
    \xi_n^{X_n} \le \zeta_n \quad\text{a.s., \qquad for all $n$.} 
  \end{equation}
  For $y_0>0$, define 
  \begin{displaymath}
    u(y_0)={\bf E}[{\bf I}(\zeta_1>y_0)\zeta_1], \qquad
    v(y_0)=-\int_{\cal X}{\bf E}[{\bf I}(\zeta_1\le{}y_0)\xi_1^x]\pi(dx).
  \end{displaymath}
  Observe that $u(y_0)\to0$ as $y_0\to\infty$ and, by the
  conditions~\eqref{absrdrift} and \eqref{rdrift}, $v(y_0)\to{}a$ as
  $y_0\to\infty$.  Choose $y_0$ sufficiently large and $K>0$ such that
  \begin{equation}\label{uv}
    u(y_0)<{\bf P}(\zeta_1>y_0)K<v(y_0).   
  \end{equation}
  We define the required i.i.d.\ sequence~$\{\psi_n\}_{n\ge1}$ by
  \begin{displaymath}
    \psi_n = {\bf I}(\zeta_n>y_0)(\zeta_n-K).
  \end{displaymath}
  It follows from the construction of this sequence, and in particular
  from \eqref{sw}, \eqref{uv} and the definition of $u(y_0)$, that it
  satisfies the conditions~\eqref{psicond1} and \eqref{psicond2} of
  Lemma~\ref{lem4} with $L_1=L$.  Define also, for each $n$ and for
  each $x$,
  \begin{displaymath}
    \eta^x_n = {\bf I}(\zeta_n\le{}y_0) \xi^x_n + {\bf I}(\zeta_n>y_0) K,
  \end{displaymath}
  The random variables $\eta^x_n$ are bounded above by $b=\max(y_0,K)$.
  Further, by \eqref{uv} and the definition of $v(y_0)$,
  \begin{displaymath}
    \int_{\cal X} {\bf E}\eta^x_1\, \pi(dx) < 0.
  \end{displaymath}
  It now follows, using also the condition~\eqref{tautail} of the
  theorem, that the sequence~$\{\{\eta^x_n\}_{x\in\cal{}X}\}_{n\ge1}$
  and $b$ as given above satisfy the conditions of Lemma~\ref{lem3},
  and so also of Lemma~\ref{lem4}.
  
  The condition~\eqref{xipsieta} follows on observing that, from
  \eqref{xibdd2}, for all $x$ and for all $n$,
  \begin{align*}
    \xi^x_n
    & = {\bf I}(\zeta_n>y_0)(\xi^x_n - K)
    + {\bf I}(\zeta_n\le{}y_0) \xi^x_n + {\bf I}(\zeta_n>y_0) K\\
    & \le \psi_n + \eta^x_n.
  \end{align*}
  Finally, it is not difficult to see that the random variables
  $M^\psi$ and $M^\eta$ (defined as in the statement of
  Lemma~\ref{lem4}) are independent (although the sequences
  $\{\psi_n\}$ and $\{\eta^{X_n}_n\}$ of which they are the maxima are
  \emph{not} independent!).  The required conclusion~\eqref{prop} now
  follows from Lemma~\ref{lem4}.
\end{proof}

\begin{proof}[Proof of Theorem~\ref{thup}]
  We again use Lemma~\ref{lem4}.  It follows from the conditions of
  the theorem that we may take $b$ such that
  \begin{equation}\label{sandwich}
    {\bf E}\zeta < b < {\bf E}_{\pi}b^{X}
  \end{equation}
  and satisfying \eqref{tautail}.  It follows also from \eqref{zbd}
  that there exists $L_1>0$ such that
  \begin{equation}\label{zbd2}
    {\bf{}P}(\zeta>y)\le{}L_1\overline{F}(y)
  \end{equation}
  for all $y$.  Further, we may define random variables
  $\{\xi^x\}_{x\in {\cal X}}$, $\zeta$, and $\{b^x\}_{x\in{\cal X}}$
  in such a way that $\zeta$ and the family~$\{b^x\}_{x\in {\cal X}}$
  are independent and, for all $x$,
  \begin{equation}\label{xzb}
    \xi^x \leq \zeta - b^x \quad \mbox{a.s.}.
  \end{equation}
  For $n=1,2,\ldots$, let $\{ \xi^x_n, b^x_n, \zeta_n \}$ be
  i.i.d.\ copies of $\{ \xi^x, b^x, \zeta \}$, such that these
  sequences are jointly independent of the process~$X$. Define, for
  all~$n$,
  \begin{equation}\label{pedef}
    \psi_n = \zeta_n-b, \qquad\qquad
    \eta^x_n = b - b^x_n,  \quad x \in \cal X.
  \end{equation}
  Then it is easy to check, from \eqref{sandwich}--\eqref{pedef} and the
  condition~\eqref{tautail} and the independence assumption of the
  theorem, that all the conditions of Lemma~\ref{lem4} are satisfied.
  The sequences~$\{\psi_n\}_{n\ge1}$ and
  $\{\{\eta^x_n\}_{x\in\cal{}X}\}_{n\ge1}$ and the constants~$L_1$ and
  $b$ of that lemma are as given here.  We thus have the required
  result.
\end{proof}

\begin{proof}[Proof of Theorem~\ref{thw}]
  We again give the proof in the case $d=1$.  It follows
  straightforwardly from the regenerative structure of $X$, the
  condition~${\bf E}\tau<\infty$, and the condition~(C2) that the
  random vectors
  \begin{align*}
    Y_0 & = \{\tau_0 ; W_1, \ldots , W_{T_0} \},\\
    Y_n & = \{\tau_n ; W_{T_{n-1}+1}, \ldots , W_{T_n} \}, \qquad n\ge1,
  \end{align*}
  form a Harris ergodic Markov chain (see, for example, \cite{HT}).
  Then, since $d=1$, it is again straightforward that $W_n$ converges
  in the total variation norm to a distribution on $\Rp$ which is
  independent of that of $Y_0$.  Now let
  $\tilde{X}=\{\tilde{X}_n\}_{-\infty<n<\infty}$ be the corresponding
  stationary version of the process~$X$ indexed over the entire set of
  integers, and similarly extend the i.i.d.\ sequence of
  families~$\{\{\xi^x_n\}_{x\in\cal X}\}_{n\ge1}$ to
  $\{\{\xi^x_n\}_{x\in\cal X}\}_{-\infty<n<\infty}$.  Let
  $\{\tilde{W}_n\}_{n\ge0}$ (with $\tilde{W}_0\equiv0$ as usual) be
  the corresponding version of the process~$\{W_n\}_{n\ge0}$.  It
  follows from the recursion~\eqref{wrec} that
  \begin{displaymath}
    \tilde{W}_n =
    \max\left(0,\xi^{\tilde{X}_n}_n,
      \xi^{\tilde{X}_n}_n+\xi^{\tilde{X}_{n-1}}_{n-1},
      \dots, \xi^{\tilde{X}_n}_n+\dots+\xi^{\tilde{X}_1}_1\right)
  \end{displaymath}
  which, by stationarity, has the same distribution as 
  \begin{displaymath}
    \max\left(0,\xi^{\tilde{X}_{-1}}_{-1},
      \xi^{\tilde{X}_{-1}}_{-1}+\xi^{\tilde{X}_{-2}}_{-2},
      \dots, \xi^{\tilde{X}_{-1}}_{-1}+\dots+\xi^{\tilde{X}_{-n}}_{-n}\right).
  \end{displaymath}
  Thus, for any $y$, $\lim_{n\to\infty}{\bf P}(W_n>y)$ and
  $\lim_{n\to\infty}{\bf P}(\tilde{W}_n>y)$ both exist and are equal
  to ${\bf P}(M^->y)$ where
  \begin{displaymath}
    M^- =
    \sup\left(0,\xi^{\tilde{X}_{-1}}_{-1},
      \xi^{\tilde{X}_{-1}}_{-1}+\xi^{\tilde{X}_{-2}}_{-2},
      \dots\right).
   \end{displaymath}
   The required result now follows from the application of
   Theorem~\ref{thsu}, \ref{thexp} or \ref{thup} as appropriate, in
   each case with $B=\cal X$, to the time-reversed version of the
   stationary process~$\{\tilde{X}_n,\xi^{\tilde{X}_n}_n\}$.  However,
   under the conditions of Theorem~\ref{thexp} or Theorem~\ref{thup},
   we must also verify the required condition on $\tau^-_0$, defined
   to be the time of the first regeneration at or after time~$0$ in
   the reversed process $X^-=\{X^-_n\}_{n\ge0}$ given by
   $X^-_n=\tilde{X}_{-n}$.  Standard renewal theory shows that the
   distribution of $\tau^-_0$ is given by
   \begin{displaymath}
     {\bf P}(\tau^-_0\ge n)
     = \frac{1}{{\bf E}(\tau)} \sum_{k=n+1}^\infty{\bf P}(\tau\ge k),
     \qquad n = 0,1,\dots
   \end{displaymath}
   An easy calculation, analogous to that of the derivation of
   Property~\ref{p:com} of Section~\ref{sec:properties}, now gives
   that, if $b>0$ is such that ${\bf{}P}(b\tau>y)=o(\overline{F}(y))$
   as $y\to\infty$, then ${\bf{}P}(b\tau^-_0>y)=o(\overline{F}^s(y))$
   as $y\to\infty$.  Thus, in each case, the required condition on
   $\tau^-_0$ follows from the assumed condition on $\tau$.

   The modifications for the case of general~$d$ are again routine.
\end{proof}

%\section{Examples and counterexamples}

%\thebibliography

\end{document}